\documentclass[12pt]{amsart}

\newcounter{defcounter}
\usepackage{enumitem,latexsym,amsthm}
\usepackage{amssymb,latexsym,eufrak,amsmath,amscd,graphicx}
  \usepackage[all]{xy}
  \usepackage{comment}
  \usepackage{pdfsync}
\numberwithin{equation}{section}
 
\theoremstyle{plain}
\newtheorem{theorem}{Theorem}
\newtheorem{proposition}[theorem]{Proposition}

\newtheorem{lemma}[theorem]{Lemma}

\newtheorem{proposition.definition}[theorem]{Proposition/Definition}

\newtheorem{theoremalpha}{Theorem}
\newtheorem{corollaryalpha}[theoremalpha]{Corollary}

\theoremstyle{definition}

\newtheorem{definitionalpha}[theoremalpha]{Definition}

\newtheorem{remark}[theorem]{Remark}
\newtheorem{example}[theorem]{Example}

\newtheorem{problem}[theorem]{Problem}

\newcommand{\GGGG}{\mathbf{G}}
\newcommand{\lra}{\longrightarrow}

\newcommand{\noi}{\noindent}
\newcommand{\PP}{\mathbf{P}}

\newcommand{\NN}{\mathbf{N}}

\newcommand{\CC}{\mathbf{C}}
\newcommand{\QQ}{\mathbf{Q}}

\newcommand{\OO}{\mathcal{O}}

\newcommand{\eps}{\varepsilon}

\newcommand{\Image}{\textnormal{Im}}
\newcommand{\HH}[3]{H^{{#1}} \big( {#2} , {#3}
\big) }

\newcommand{\hh}[3]{h^{{#1}} \big( {#2} , {#3}
\big) }

\newcommand{\coker}{\textnormal{coker}}

\newcommand{\codim}{\textnormal{codim}}

\newcommand{\pr}{\prime}

\newcommand{\lin}{\equiv_{\text{lin}}}

\newcommand{\dra}{\dashrightarrow}

\newcommand{\Supp}{\operatorname{Supp}}

\newcommand{\Linser}[1]{| \mspace{1.5mu} {#1}
\mspace{1.5mu} |}
\newcommand{\linser}[1]{\Linser{  {#1}  }}

\newcommand{\gon}{\textnormal{gon}}
\newcommand{\pro}{{pr}}

\newcommand{\ev}{\textnormal{ev}}

\newcommand{\CCC}{\mathcal{C}}
\newcommand{\VVV}{\mathcal{V}}
\newcommand{\vol}{\operatorname{vol}}

\newcommand{\irrdeg}{\textnormal{irr}}

\newcommand{\pconngon}{p\text{-}\textnormal{conn.\,gon}}

\newcommand{\pconngenus}{p\text{-}\textnormal{conn.\,genus}}

\numberwithin{theorem}{section}

\begin{document}

\title{Connecting Families of Curves}

\author{Nathan Chen}
\address{Instituto de Matemática Pura e Aplicada (IMPA)}\email{nathanchenmath@gmail.com}

\author{Robert Lazarsfeld}
\address{Stony Brook University and University of Pennsylvania}
\email{robert.lazarsfeld@stonybrook.edu}

\author{Federico Moretti}
\address{Stony Brook University}
\email{federico.moretti@stonybrook.edu}
\begin{abstract}
A fundamental theorem of  Koll\'ar, Miyaoka and Mori asserts that on a rationally connected variety, any finite collection of points can be connected by a rational curve. Motivated by this, it is natural to ask about  the geometry and numerics of families of curves passing through $p\geq 2$ general points of an arbitrary  smooth projective variety $X$ of dimension $n$. Assuming that $\kappa(X)\geq 0$, we show that the least genus $g$ of curves in such a family satisfies
\[
g\ \geq\ (n-1)p-(n-2);
\]
this strengthens a result of Arapura and Archava. By contrast, the least gonality of curves in a $p$-connecting family stabilizes when $p\gg 0$. We characterize its limiting value geometrically in terms of the smallest degree of a generically finite covering of a rationally connected variety by a variety dominating $X$. As an illustration, we compute the joint asymptotics in $p$ and $d$ of the minimal $p$-connecting genus for hypersurfaces of degree $d$. Finally, we consider higher-dimensional connecting families: when $X$ is of general type, we establish linear bounds on the minimal canonical volumes of families of $k$-folds passing through $p$ general points of $X$.	
\end{abstract}

\maketitle

\setlength{\parskip}{.13in minus .03in}

     \section*{Introduction}
 
The purpose of this paper is to study the geometry and numerics of families of curves connecting many points of a projective variety.  

Let $X$ be a smooth complex projective variety of dimension $n\ge 2$. Recall that $X$ is \textit{rationally connected} if two general points of $X$ can be joined by an irreducible rational curve. A basic result of   Koll\'ar--Miyaoka--Mori \cite{KMM} asserts that once this condition holds, one can find a rational curve passing through any finite number of points of $X$. Motivated by this, it is natural to ask what one can say about families of curves connecting many points of $X$ when it is not rationally connected.

A first theorem along these lines  is due to  Arapura and Archava \cite{Arapura.Archava}, who consider a family 
of irreducible curves generically covering $X$. Assuming that $\kappa(X) \ge 0$, they show that if there are curves in the family passing through $p \ge 2$ general points of $X$ then the curves must have genus $g  \ge p$. In addition,    if $X$ has general type then $g > p$.\footnote{By the genus or gonality of an irreducible curve $C$, we mean the corresponding invariant of its smooth model.}  

The result of Arapura--Archava suggests
\begin{definitionalpha}\label{Intro.Def.A}
Fix $p \ge 2$. We define the \textit{$p$-connecting genus} and \textit{$p$-connecting gonality} of $X$ to be the least genus or gonality  of a family of curves containing members passing through $p$ general points of $X$. In other words,
\begin{align*} 
\pconngenus(X) \ &= \ \min \Bigg  \{ g \ge 0 \  
\Bigg | \parbox{2.7in}{\begin{center} General points $x_1, \ldots, x_p\in X$ can be connected by an irreducible curve $C \subseteq X$ with  $g(C) = g$. \end{center}}
\Bigg \}; 
\\ \\
\pconngon(X) \ &= \ \min \Bigg  \{ c \ge 1 \ 
\Bigg | \parbox{2.7in}{\begin{center} General points $x_1, \ldots, x_p\in X$ can be connected by an irreducible curve $C \subseteq X$ with  $\gon(C) = c$. \end{center}}
\Bigg  \}.   \end{align*} 
\end{definitionalpha}

\noi Thus
\[ X \text{\ is rationally connected} \ \Longleftrightarrow \ \pconngenus(X) = 0\ \Longleftrightarrow \ \pconngon(X) = 1\] 
for all $p \ge 2$, and the theorem of Arapura--Archava is that if $\kappa(X) \ge 0,$ then 
\[ \pconngenus(X) \ \ge \ p.    \]
Clearly these are birational invariants of $X$. 

For surfaces the statement of Arapura--Archava is 
asymptotically the best possible: if $n = 2$ and $\kappa(X) \ge 0,$ then
\[
\liminf_{p\to\infty} \
\frac{\pconngenus(X)}{p}=1,
\]
  (Example \ref{connecting.genus.sfs}). However  our first main result shows that in higher dimensions, one can do better.
  
\begin{theoremalpha} \label{Intro.Thm.B}
	Assume that $\kappa(X) \ge 0$. Then
	\[
	\pconngenus(X) \, \ge \, (n-1)\cdot p - (n-2).
	\]
	Moreover if $X$ is of general type, then strict inequality holds. 
\end{theoremalpha}
\noi This is established by studying the positivity of the normal bundle to a generic member of a $p$-connecting family. There are examples where the inequality in the theorem is optimal or close to it (Example \ref{CY.Sharp.Example}), but we do not expect it to be so in general, even asymptotically. However elementary arguments show that $\pconngenus(X)$ is bounded above linearly in $p$ (Example \ref{Linear.Growth.Connecting.Genus}).

Turning to  $p$-connecting gonality, the picture is very different. In fact, fix a branched covering
\[  f : X \lra \PP^n  \ \ \text{of some degree} \ 
 \delta.  \tag{*} \]
 Because we can join any number of points of $\PP^n$ by a rational curve, 
  one   shows that then $\pconngon(X) \le  \delta$ for all $p \ge 2$. Since in any event the $p$-connecting gonality of $X$ is non-decreasing in $p$, and bounded above thanks to (*), it follows that 
 \[
 \text{the function }  p\, \mapsto \, \pconngon(X) \  \text{is eventually constant}.
 \]
Our second set of results gives a geometric explanation for the final value of this invariant. 

To begin with, we prove
\begin{theoremalpha} \label{Intro.Eventual.Gonality.Thm}
Fix an integer $k \ge 1$ and suppose that $p$ general points of $X$ are connected by a curve that carries a pencil of degree $k$. If 
\[
p \ge  kn,
\]
then there exists a diagram
\[
\xymatrix
 @C=3em @R=2em {Y \ar[r]^{\sigma}\ar[d]_{\tau} & X\\
R,
}
\]
where $Y$ is a variety dominating $X$, $R$ is rationally connected, and $\tau$
is a generically finite covering of degree $\le k$.
\end{theoremalpha}
\noi This is established by analyzing  the MRC fibration associated to  the variety of effective zero-cycles on $X$ of degree $k$ arising from the pencils on the connecting curves.

The existence of such a diagram  implies that $\pconngon(X)\le k$ for every $p$. Therefore we find:
\begin{corollaryalpha} \label{Rat.Conn.Covering.Ex}
	For $p \gg 0$, the $p$-connecting gonality of $X$ is given by \vspace{4pt}
\[
\pconngon(X)
=
\min \Bigg\{
c \ge 1
\ \Bigg|\
\parbox{3.8in}{
\centering
\textnormal{$X$ is dominated by a variety $Y$ that is a
$c$-fold generically finite cover of a rationally connected variety}
}
\Bigg\}.
\]\end{corollaryalpha}
\noi In other words, the final value of $\pconngon(X)$ does in fact have  a natural geometric meaning. 
 
By way of example, we next compute the asymptotics of these invariants for  hypersurfaces of large degree.
\begin{theoremalpha} \label{Connecting.Invariants.of.Hypersurfaces}
Fix $n\geq 2$, and let
$X_d\subseteq\PP^{n+1}$ be a smooth hypersurface of degree $d$
and dimension $n$. Then:
\begin{enumerate}
    \item[$($i$)$.] For every  $p\geq 2$,
    \[
        \pconngon(X_d)\, \sim d\,
        \qquad\text{as }d\to\infty.
    \]

    \vspace{4pt}

    \item[$($ii$)$.] There is a function
   $
        r_n(p)\asymp_n p^{\frac{1}{n+1}}
   $
    such that
    \[
        \pconngenus(X_d) \, 
        \in\, 
        \Theta_n\left(
            p^{\frac{n-1}{n+1}}\cdot d^2
        \right)
    \]
    whenever $p\geq 3$, $d \ge n+3$ and $r_n(p) < d$.
\end{enumerate}
\end{theoremalpha}\noi (The  definition of $r_n(p)$ and the precise meaning of the joint asymptotics in (ii) are spelled out in the statement of Theorem \ref {Precise.Hypersurface.Thm} below.)

Finally, in the spirit of  a coda we consider briefly the analogous questions for families of $k$-dimensional subvarieties ($1 \le k \le n-1$) that connect general points of $X$. Suppose then that $V \subseteq X$ is an irreducible subvariety of dimension $k$ that moves in a family $\VVV = \{ V_t \}_{t \in S}$  passing through $p$  general points. The $p$-connecting genus of a family of curves is naturally replaced in this context by the canonical volume $\vol(V)$, i.e., the volume of the canonical bundle of any smooth model of $V$.\footnote{The definition of this invariant is recalled in  item (2) in the Notations and Conventions below. }
\begin{theoremalpha} \label{k-dim.p-conn.Thm}
	Assume that $X$ is of general type. The minimal volume of a
$p$-connecting family of subvarieties of dimension $k$ grows linearly
in $p$.\footnote{Since the volume of a variety of dimension $\ge 3$ can
be non-integral, it is not a priori obvious that a minimum exists.
However, it's a (highly non-trivial) fact that the possible values of
the canonical volume of smooth projective varieties of given dimension
form a discrete set, with the positive values bounded away from zero:
see \cite[Remark~1.5]{HMX}.} Specifically, for $p \gg 0$ any such family satisfies 
	\[ \vol(V_t) \ \ge \ a_{k,X} \cdot p   \]
for some real number $a_{k,X} > 0$ depending only on $k$ and $X$. Moreover, there exist families with \[ \vol(V_t) \ \le \ A_{k,X}\cdot p\] for another constant $A_{k,X}$ with the same dependencies.
\end{theoremalpha}
\noi (It is not clear to us what might be the right generalization of $p$-connecting gonality for families of higher-dimensional subvarieties: see Remark \ref{Higher.Dim.Gonality}). We also establish the analogue of Theorem \ref{Connecting.Invariants.of.Hypersurfaces} for $k$-dimensional connecting families of  subvarieties of hypersurfaces; along the way we record a Castelnuovo-type bound on the canonical volume of an irreducible non-degenerate subvariety of projective space (Proposition \ref{cast}). 

  Besides the theorem of Arapura--Archava
discussed above, questions concerning varieties connected by curves of
fixed positive genus were studied by Gounelas \cite{Gounelas}, and, in
the case of elliptic curves, by Lazi\'c--Peternell
\cite{Lazic.Peternell}. The  two-point connecting gonality was
introduced in \cite{BDELU}, and has subsequently been investigated for
hypersurfaces, Fano surfaces, and symmetric products of curves
\cite{BCFS.Gonality,Gounelas.Kouvidakis,Bastianelli.Picoco}. Closely related
bounds on the geometric genus of members of moving families appear in
\cite{LMP}.  From a numerical point of view, Lehmann's mobility count
\cite{Lehmann.Mobility} measures the number of general points that can
be imposed on members of a family of cycles, but this has a quite different flavor. Apropos of coverings by higher-dimensional subvarieties, we point to Voisin's interesting paper  \cite{Voisin.Nori}.

Concerning the organization of the paper, \S 1 is devoted to the connecting genus. The results for $p$-connecting gonality appear in \S 2, while the hypersurface asymptotics occupy  \S 3. Finally, \S 4 studies families of higher-dimensional subvarieties passing through many general points.

\subsection*{Notation and Conventions.} 
(1).  We work throughout over the complex numbers.

\noi (2). Given an irreducible projective variety $X$ of dimension $n$, the $m^{\text{th}}$ plurigenus of $X$ is defined to be
\[
P_m(X) \ = \ \dim \, \HH{0}{X^\pr}{mK_{X^\pr}},
\]
where $X^\pr$ is any resolution of $X$. The \textit{canonical volume} of $X$ is
\[
\vol(X) \ = \ \lim_{m \to \infty} \ \frac{P_m(X)}{m^n/n!}.
\]
Thus $X$ is of general type if and only if $\vol(X) > 0$, and if $C$ is an irreducible curve of geometric genus $g \ge 2$, then $\vol(C) = 2g-2$. We refer to \cite[\S 2.2.C]{PAG} for further information about the volume of a big divisor. 

\noi (3). We use the usual asymptotic notation for positive numerical
functions. Thus $a(t)\sim b(t)$ means that
$
a(t)/b(t)\longrightarrow 1
$
as $t \to \infty$, while
\[
a(t)\asymp b(t)
\]
means that there are constants $c,C>0$ such that
\[
c \cdot b(t)\, \le \,  a(t)\, \le \,  C\cdot b(t)
\]
uniformly in the indicated range of the variables. A subscript, as in
$a(t)\asymp_n b(t)$, indicates that the constants are allowed to depend on
the indicated parameter. We use
\[
a(t)\, \in\, \Theta_n\big(b(t)\big)
\]
synonymously with $a(t)\asymp_n b(t)$. Finally, the symbols $O(\,\cdot\,)$
and $o(\,\cdot\,)$ have their standard meanings, with subscripts again
indicating the  parameters on which the implied constants are allowed to
depend.

\subsection*{\textbf{Acknowledgements and AI use.}}   We are grateful to Pietro Pirola, Mihnea Popa, Andr\'es Rojas, and Claire Voisin for valuable conversations and suggestions. The present paper originated in discussions that took place while the second author was  a Benedict Gross Distinguished Visitor at the Harvard University 
Mathematics Department, and he thanks the department for its hospitality. 

In the later stages of this project, the authors  profited from the use of ChatGPT 5.6 Sol as an interactive research tool. Besides helping us to check and clean up various arguments,  this system suggested the application of Castelnuovo's bound in the proof of Theorem \ref {Connecting.Invariants.of.Hypersurfaces}, and it  contributed Lemma \ref{RC.Cover.Pullback.Lemma} as well as several of the results in \S \ref {Higher.Dim.Coverings}.   Naturally the present authors take full responsibility for the presentation and correctness of the material here.

\section{Connecting genus} \label{Section.Connecting.Genus}

This section is devoted to proving a lower bound on the $p$-connecting genus  that strengthens the inequality stated in  the Introduction. The main result is Theorem \ref{Main.Genus.Thm,Section.1}. 

We start  by fixing our set-up and notation. As in the Introduction, $X$ is a smooth complex projective variety of dimension $n \ge 2$. By a \textit{family of curves} on $X$ we mean a diagram
\begin{equation} \label{Fam.of.Curves}
\vcenter{\hbox{$
\xymatrix{
\CCC \ar[r]^{F} \ar[d]_\pi &X \\ S}
$}
}	
\end{equation}
where $S$ is a non-singular quasi-projective variety and $\pi$ is a smooth projective morphism of relative dimension $1$ with connected fibres. Given $ t \in S$ we denote by
\begin{equation} \label{Fibrewise.Curve}
 f_t : C_t \lra X  
 \end{equation}
the fibrewise data over $t$. We will also assume that $f_t$ is birational over its image for general $t\in S$: this is equivalent to asking that the evident map $\CCC \lra X \times S$ be generically one-to-one.  The genus or gonality of the family is the genus or gonality of its general member $C_t$.

Fix next an integer $p \ge 2$. A family \eqref{Fam.of.Curves} determines in the natural way a morphism
\[
F^p \ : \ \CCC^p_{/S} \lra X^p 
\]
from  the $p$-fold fibre product of $\CCC$ over $S$ to the $p$-fold product of $X$. Set-theoretically, $F^p$ sends a $p$-tuple of points on some $C_t$ to their images under $f_t$. We say that the family \eqref{Fam.of.Curves} is \textit{$p$-connecting} if $F^p$ is dominant. Note that this forces
\[ \dim S \, \ge \, p \cdot( n-1).       \]
As in the Introduction, the $p$-connecting genus or gonality of $X$ is defined to be the minimal genus or gonality of a $p$-connecting family. 

It will be important to understand the picture infinitesimally. In the situation of \eqref{Fam.of.Curves}, fix a general point $t \in S$, and write
$N = N_{f}$ for the normal sheaf to $f= f_t$ along $C =C_t$. Thus $N$ sits in an exact sequence
\begin{equation}
\label{Tangent.Exact.Seq}	
0 \lra T_C \lra f^* T_X \lra N \lra 0
\end{equation}
of sheaves.
Denote by $N^\pr = N^{\text{tf}}$ the torsion-free quotient of $N$, so that $N^\pr$ is a vector bundle of rank $n-1$ on $C$. The derivative of $F$ determines a  map $dF: N_{C_t/\CCC} \lra N^\pr$. Observing that $N_{C_t/\CCC} = T_tS \otimes_\CC \OO_C$, we arrive upon taking global sections at the Kodaira--Spencer-type  characteristic homomorphism
\[
\chi : T_tS \lra \HH{0}{C}{N^\pr}
\]
determined by \eqref{Fam.of.Curves}.

 The $p$-connecting hypothesis  enters the discussion through the following essential observation. 
 
 \begin{lemma} \label{GGG.Lemma}
Assume that the family \eqref{Fam.of.Curves} is $p$-connecting. Fix a general point $t \in S$ and as above put $C = C_t$ and $N^\pr = N^{\textnormal{tf}}$. Then for any $p-1$ general points $q_1, \ldots, q_{p-1} \in C$, the vector bundle
\[
N^\pr(-q_1 - \ldots -q_{p-1})
\]
is generically globally generated. 
\end{lemma}
\noi In other words, the global sections of the bundle in question span a subsheaf of full rank $n-1$. 

\begin{proof}
By the theorem on generic smoothness, $F^p$ is smooth on a Zariski-open subset of $\CCC^p_{/S}$.
Consider now  a point 
\[u = (t; q_1, \ldots, q_p) \ \in \  \CCC^p_{/S} \]
in this smooth locus.
Much as in the previous paragraph, the derivative of $F^p$ determines a surjective map
$ dF^p : 
T_t S\, \lra \, \oplus_{i = 1}^p \, (N^\pr)_{q_i}. 
$
Write
\[
\Lambda \ =_\text{def} \ \ker \Big( T_t S\, \lra \, \oplus_{i=1}^{p-1}(N^\pr)_{q_i} \Big)
\]	
for the kernel of the first $(p-1)$ components of $dF^p$.  Then $\Lambda$  maps via the characteristic homomorphism $\chi$  to  a subspace of $\HH{0}{C}{N^\pr(-q_1 - \ldots -q_{p-1})}$, and it follows from the surjectivity of $dF^p$ that this subspace generates $N^\pr$ at $q_p$ and hence generically. In other words, starting at a point $u = (t,q_1, \ldots, q_p)$ at which $F^p$ is submersive, we see that $N^\pr(-q_1 - \ldots - q_{p-1})$ is generically globally generated. But for general $t \in S$, $F^p$ is smooth on an open subset of $C_t^p$, and the Lemma follows. 
 \end{proof}

We now come to the main result of this section.
  \begin{theorem} \label{Main.Genus.Thm,Section.1}
 Let $f : C \lra X$ be the general member of a $p$-connecting family of curves as in \eqref{Fam.of.Curves}, with $g(C) = g$.  Assume that $\kappa(X) \ge 0$. Then
 \begin{equation} \label{Main.genus.Ineq}
2g \, - \, 2 \ \ge \ 2(n-1)(p-1) \ + \ \deg_C ( f^* K_X). 
 \end{equation}
 \end{theorem} 
\noi Observe that if $\kappa(X) \ge 0$, then $K_X$ is represented by an effective $\QQ$-divisor, and since $C$ moves in a covering family it follows that $\deg_C(f^* K_X) \ge 0$. Thus \eqref{Main.genus.Ineq} implies the first inequality appearing in Theorem \ref{Intro.Thm.B}. Similarly, if $X$ has general type, then $\deg_C (f^*K_X) > 0$ and the second statement follows.

 \begin{proof} [Proof of Theorem \ref{Main.Genus.Thm,Section.1}]
 We start with a $p$-connecting family \eqref{Fam.of.Curves}.
 Fix $t \in S$ for which the conclusion of Lemma \ref{GGG.Lemma} holds,  put $L = \det N^\pr$, and set $r = n-1$. We write $g = g(C)$ for the genus of the curves in our family. Note to begin with that the Lemma guarantees that 
 $\det \big(N^\pr(-q_1 - \ldots - q_{p-1})\big)$ is effective for general choices of the $q_i$, which is to say:
 \[
  \hh{0}{C}{L(-rq_1 - \ldots - rq_{p-1})}\ \ge \ 1.
\]
 But now recall that if $A$ is any line bundle on $C$, and if $q \in C$ is a general point, then the vanishing sequence for $A$ at $q$ is the generic one, and hence 
 \[ h^0\big ( A( -r\cdot q)\big) \ = \ \max \{ h^0(A) - r, 0 \}. 
 \]   Starting with $L$ we apply this observation with each  $q_i$  chosen generically in turn.  We then find 
\begin{equation} \label{Lower.Bound.h0.L}
 h^0(L) \ \ge \ 1 \, + \, r  (p \, - \, 1).
\end{equation}

 We next use adjunction  to deduce an upper bound involving $L$. Write
\[
e \ = \ \deg_C( f^* K_X) ;
\]
we have already observed that the non-negativity of $\kappa(X)$ implies that $e \ge 0$. 
It follows from \eqref{Tangent.Exact.Seq} that
 \begin{equation} \label{Upper.Bound.detN}
 \deg(L) \ \le \ \deg(N) \ = \ (2g-2) \, - \, e. \end{equation}
 We will now see  that \eqref{Lower.Bound.h0.L} and \eqref{Upper.Bound.detN}
 together formally imply the theorem.
 
 In fact suppose first that $L$ is special, i.e., $h^1(L) \ne 0$. Then Clifford's theorem gives:
 \[ \deg(L) \ \ge \ 2\cdot \big(h^0(L) -1 \big), \] 
 and we arrive directly at \eqref{Main.genus.Ineq}. On the other hand, suppose that $L$ is non-special. Then $h^0(L) = \deg(L) + 1 - g$, and hence
 \[
 (g-1) \, - \, e \ \underset{\eqref{Upper.Bound.detN}}\ge  \ \deg(L) \, + \, 1 \, - \, g \ = \ h^0(L) \ \underset{\eqref{Lower.Bound.h0.L}}\ge \ 1 \, + \, r(p-1). 
 \]
 This leads to the inequality
 \[
 (g\, - \, 1 ) \ \ge \ r(p\, - \, 1) \, + \, e \,+ \, 1,
 \]
 which is stronger than \eqref{Main.genus.Ineq}.   \end{proof}
 
 \begin{remark}[\textbf{Boundary examples are hyperelliptic}]
It is amusing to observe that if $X$ is of general type, and if equality holds in \eqref{Main.genus.Ineq}, then the connecting curves $C_t$ must be hyperelliptic.  In fact, keeping the notation of the proof just completed, equality can only occur when  $L$ is  special and $\deg(L) = 2 \cdot \big( h^0(L) -1 \big)$. So it follows from Clifford's theorem that either $L = \OO_C, L = K_C$ or $C$ is hyperelliptic. But $1 \le \deg(L) \le (2g-2) -e$, so the first two possibilities are ruled out. Example \ref{CY.Sharp.Example} below shows that such hyperelliptic connecting families  actually occur.  We note   one could expect to strengthen some of the inequalities by taking into account invariants such as the Clifford index of the connecting families. \qed
 \end{remark}
 
 \begin{remark} [\textbf{A lower bound for \eqref{Main.genus.Ineq}}] In view of   \eqref{Main.genus.Ineq}, it is of interest to find lower bounds for the degree $\deg_C( f^* K_X)$. Here is one such:
 \begin{quote}
 Assume that $X$ is of general type, fix $\eps > 0$ and let $f : C \lra X$ be a general member of a $p$-connecting family. If $p\gg_\eps 0$, then
 \[
 \deg_C (f^*K_X) \  \ge \ (1-\eps) \, \cdot \,  \sqrt[\raisebox{0.4ex}{\scriptsize $n$}]{ \frac{\vol(X)}{n!} }\cdot p^{\frac{n-1}{n}  }   . 
 \]
 \end{quote}
 In fact, let $m = m(X, p)$ be the smallest integer such that $h^0(mK_X) \ge p$. Then $p$ general points  $x_1, \ldots, x_{p} \in X$ impose independent conditions on $H^0(mK_X)$. Therefore  we can find a section $\sigma\in H^0(mK_X)$ vanishing at $x_1, \ldots, x_{p-1}$ but not at $x_p$. It follows that if  $f : C \lra X$ connects $x_1,  \ldots, x_p$, then 
 \[ \deg_C (f^* K_X) \ \ge \  \frac{p-1}{m(X,p)}.\]  But by definition
$ \vol(X) =\lim_{m \to \infty} \tfrac{ h^0(mK_X)}{m^n / n!}$, so $m(X,p)\approx \left( \frac{n!p}{\vol(X)}\right)^{\frac{1}{n}}$. \qed
\end{remark}

We conclude with some  examples and questions.

\begin{example} [\textbf{Surfaces}] \label{connecting.genus.sfs} Let $X$ be a smooth projective surface with $\kappa(X) \ge 0$. Then \[\pconngenus(X) \ge p\] thanks to \cite{Arapura.Archava} or Theorem \ref{Intro.Thm.B}. Let us show that one can find $p$-connecting families of curves $C$ for which the ratio $\frac{g(C)}{p}$ comes arbitrarily close to $1$. In fact, fix a very ample line bundle $A$, and consider a  positive integer $m$. Then
\[
 r \, =_\text{def} \,  r(mA) \ = \ \frac{ (A^2)}{2} m^2 \, + \, O(m)
\]thanks to Riemann-Roch, and we can find smooth curves $C \in |mA|$ passing through any $ r$ general points (Lemma \ref{Bertini.Lemma} below). On the other hand, by the adjunction formula
\[
g(C) \ = \ \frac{ (A^2)}{2} m^2 \, + \, O(m).
\]
Taking $m \gg 0$ yields the advertised curves.  \qed
  \end{example}

\begin{example} [\textbf{Some   double covers}] \label{CY.Sharp.Example} Here are some examples where Theorems \ref{Intro.Thm.B} and \ref{Main.Genus.Thm,Section.1} are optimal, or nearly so. Starting for simplicity with $n=3$, fix a smooth hypersurface $ B \subseteq \PP^3$ of degree $2b \ge 8$,  and consider the double covering \[ \pi : X \lra \PP^3\] branched along $B$. By a standard dimension count, when $d \ge 3$ we can find a  smooth rational curve $\Gamma = \Gamma_d$ passing through any $  2d$ general points of $\PP^3$, and typically $\Gamma$ will meet $B$ transversely. Starting with $2d$ general points on $X$, we take $\Gamma = \Gamma_d$  to pass  through their images in $\PP^3$. Then the curves 
\[C \ = \ C_d \ = \ \pi^{-1}(\Gamma)\] yield a $2d$-connecting family of genus $bd - 1$, and equality holds in the genus bound from Theorem \ref{Main.Genus.Thm,Section.1}. When $2b = 8$, so that $X$ is Calabi--Yau, we get equality in Theorem \ref{Intro.Thm.B}.	

More generally, let \(B \subseteq \PP^n\) ($n \ge 3$) be a smooth hypersurface of
degree \(2b \ge 2(n+1)\), and let
\[
\pi \colon X \longrightarrow \PP^n
\]
be the double cover branched over \(B\). With \(d \ge n\) put
\[
p_d
=
\left\lfloor
\frac{(n+1)d+n-3}{n-1}
\right\rfloor .
\]
Interpolation for rational curves \cite{AtanasovLarsonYang} yields smooth
rational curves \(\Gamma_d \subseteq \PP^n\) of degree \(d\) passing
through \(p_d\) general points, which may be taken transverse to \(B\).
Their inverse images \(C_d=\pi^{-1}(\Gamma_d)\) form a
\(p_d\)-connecting family of curves of genus
\[
g(C_d)=bd-1,
\]
while
\[
\deg_{C_d}(K_X)=2d\big(b-(n+1)\big).
\]
Writing
\[
(n+1)d+n-3=(n-1)p_d+\rho_d,
\qquad 0\leq \rho_d\leq n-2,
\]
one computes
\[
2g(C_d)-2
=
2(n-1)(p_d-1)+\deg_{C_d}(K_X)+2\rho_d.
\]
Thus the curves \(C_d\) miss equality in Theorem \ref{Main.Genus.Thm,Section.1}
by at most \(2(n-2)\), and equality holds whenever
$
(n-1)\mid 2(d-1).
$
When \(b=n+1\), so that $K_X = \OO_X$ this yields further examples
showing that Theorem \ref{Intro.Thm.B} is optimal or nearly so.  \(\qed\)

\end{example}

\begin{example} [\textbf{Linear growth of $p$-connecting genus}] \label{Linear.Growth.Connecting.Genus}
Let $X$ be an arbitrary smooth projective variety of dimension $n$. Then a construction in the spirit of the previous example shows that $\pconngenus(X)$ is bounded above by a linear function of $p$ (depending on $X$). 

In fact, assuming for simplicity of exposition that $n \ge 3$, start  with a suitably positive projective embedding $X \subseteq \PP^N$, and take a general linear projection to construct a ramified covering
\[ \pi : X \lra \PP^n \ \ ,  \ \ \deg(\pi) \, = \, \delta
\]
branched along a hypersurface $B \subseteq \PP^n$. We can assume that $\pi$ is simply ramified over a general point $y \in B$, i.e., that the monodromy of $\pi$ at $y$ is a single transposition. In this case $2b =_\text{def} \deg B$ is even, and if $\Gamma \subseteq \PP^n$ is a rational curve of degree $d$ meeting $B$ transversely, then $C = \pi^{-1}(\Gamma )$ is a smooth irreducible curve of genus
\[
g \ = \ bd - \delta + 1. 
\]
(The connectedness of $C$ is automatic thanks to the Fulton--Hansen theorem; cf.~\cite[Theorem 4.1]{FL}.) At the risk of estimating wastefully, we can use these curves to connect $p = d$ general points of $X$. In this case we find that 
\[
\pconngenus(X) \ \le \ b \cdot p - (\delta - 1) \ \ \text{for all $ p \ge 2$}.  
\]
Note that all these curves have gonality $\le \delta$ since by construction they admit a degree $\delta$ map to $\Gamma \cong \PP^1$. Thus $\pconngon(X) \le \delta$ for all $p$.\footnote{A linear upper bound on $\pconngenus(X)$ is also witnessed by suitable complete intersection curves, as in the proof of Theorem \ref{Precise.Hypersurface.Thm}. However the present approach is useful for later reference.}  \qed
\end{example}

\begin{problem}
Theorem \ref{Intro.Thm.B} implies that if $\kappa(X) \ge 0$, then
\[
\liminf_{p \to \infty} \ \frac{\pconngenus(X)}{p}\  \ge \ n - 1,
\]
but it seems very unlikely that this is always optimal. It would be interesting to get a sense for what are the possibilities for the ratio on the left. 
\end{problem}

\section{Connecting gonality}

We turn in this section to $p$-connecting gonality. As above, $X$ denotes a smooth complex projective variety of dimension $n \ge 2$.

To set the stage, 
 recall from the Introduction that as $p$ grows, the values of the $p$-connecting gonality of $X$, which are evidently non-decreasing,  stabilize to a fixed eventual value.  In fact, we observed at the end of  Example \ref{Linear.Growth.Connecting.Genus}  that these invariants are bounded above 
by the degree of a suitable branched covering $X \lra \PP^n$.\footnote{With a little more care one sees that $\pconngon(X) \le \operatorname{irr}(X)$, where the degree of irrationality $\operatorname{irr}(X)$ is by definition the least degree of a rational covering $X \dra \PP^n$. } Our goal here is to give a geometric explanation for this limiting $p$-connecting gonality.

  As one might suspect,  the eventual value of $\pconngon(X)$ does not  reflect only the degrees of coverings (or rational coverings) of projective space. This is witnessed by many examples, but  here is an interesting one: 
\begin{example} [\cite{FKP}, \cite{Moretti}] Let $(X, L)$ be a very general polarized $K3$ surface of genus $r(r+1)$  for some $r \in \NN$. Then any $p$ general points of $X$ are connected by a hyperelliptic curve, i.e., 
\[
\pconngon(X) \ = \ 2  \ \  \text{for all \, $p \ge 2$}. \tag{*} 
\]
On the other hand, a general $K3$ of large degree is not a two-fold rational  cover of $\PP^2$ (and one conjectures that the least degree of a rational covering $X \dra \PP^2$ goes to infinity as the polarization grows).

For (*), the idea is that under the stated numerical hypothesis one can construct vector bundles $E, F$ of consecutive ranks $r$ and  $r+1$ on $X$ with the properties that \[ \dim \operatorname{Hom}(E, F) \  = \ 3, \] and that every homomorphism $\sigma : E \lra F$ drops rank at two points. The incidence correspondence relating a homomorphism to its degeneracy locus then gives rise to a diagram
\[
\xymatrix{
Y \ar[d]_\tau \ar[r]^\sigma &X \\ \PP^2,
}
\]
where $\sigma$ is a branched covering of some large degree and $\tau$ is a double cover. As in the examples appearing at the end of the previous section, this implies that $\pconngon(Y)=2,$ and then the same holds for $X$. Hyperelliptic curves on a $K3$ were studied in \cite{FKP}; the perspective here is from \cite{Moretti}.  For an arbitrary $n$-fold $X$, our final picture is  similar: one constructs  dominant morphisms $Y \lra X$ where $Y$ is a generically finite cover of a rationally connected variety. \qed
\end{example}

We next turn to the proof of Theorem
\ref{Intro.Eventual.Gonality.Thm}. The basic idea is very simple. Namely,
the images on $X$ of the degree-$k$ divisors in the given pencils sweep
out a subvariety
\[
Z\ \subseteq\ \operatorname{Sym}^k(X).
\]
The rationally connected variety $R$ arises as a very general fibre
of the MRC fibration of a smooth model of $Z$, and the dimensional
hypothesis will guarantee that the variety $Y$ obtained by pulling back
to $R$ the natural map
\[
X\times\operatorname{Sym}^{k-1}(X)
   \longrightarrow \operatorname{Sym}^k(X)
\]
dominates $X$. The following paragraphs flesh out the details of this
outline.

\begin{proof}[Proof of Theorem \ref{Intro.Eventual.Gonality.Thm}]
We assume given a $p$-connecting family \eqref{Fam.of.Curves} whose general fibre
$C_t$ carries a pencil of degree $k$.

\noi \textsc{Step 1}. We start with some standard reductions. By removing the fixed part of
the pencil and decreasing $k$ if necessary, we may suppose that the
pencil on the general fibre is basepoint-free. The hypothesis $p\geq kn$
remains valid, while the desired conclusion is only strengthened, so we
continue to denote the new degree by $k$.   

After making a generically finite dominant base change on $S$, we may
suppose first that $\pi:\CCC\to S$ has a section. Consider next the
relative parameter space
\[ \nu : G^1_k(\CCC/S)\longrightarrow S
\]  of   pencils of degree $k$ on the fibres of $\pi$. Choose a component
dominating $S$, and take a generically finite multisection of this
component. After making the corresponding base change and then
shrinking $S$, we may suppose that $\nu$ also admits a section. Observe that since all the base changes in question are dominant, the
resulting family remains $p$-connecting.

The existence of a section of $\pi$ guarantees that the relative Picard
scheme carries a normalized Poincar\'e bundle. Consequently, the chosen
section of $\nu$ determines a line bundle $\mathcal A$ on
$\CCC$ of relative degree $k$, together with a rank-two subbundle
\[
U\subseteq\pi_*\mathcal A
\]
whose restriction to each fibre gives the chosen pencil. After shrinking
$S$ once more, the evaluation map
$
\pi^*U\longrightarrow\mathcal A
$
is surjective, and hence determines a degree $k$ covering
\[
h:\CCC\longrightarrow\PP_S(U).
\]
Finally, shrinking $S$ to trivialize $U$, we obtain a diagram
\begin{equation}
\label{Relative.Pencil.Diagram}
\vcenter{\hbox{%
$\xymatrix{
\CCC \ar[rr]^h \ar[dr]_\pi
    && S\times\PP^1 \ar[dl]^{\pro_1}\\
    & S
}$}}
\end{equation}
whose restriction to the fibre over $t\in S$ is the degree $k$ map
$h_t:C_t\lra\PP^1$ defined by the chosen pencil.

\noi \textsc{Step 2}. We henceforth focus on the diagram
\eqref{Relative.Pencil.Diagram}. Denote by $\CCC^{(k)}_S$ and $X^{(k)}$
the $k$-fold symmetric products of $\CCC$ over $S$ and of $X$,
respectively, and write
\[
F^{(k)}:\CCC^{(k)}_S\lra X^{(k)}
\]
for the map induced by $F$. The map $h$ in
\eqref{Relative.Pencil.Diagram} determines a morphism
\[
S\times\PP^1\lra\CCC^{(k)}_S,
\]
and we denote by
\[
\gamma:S\times\PP^1\lra X^{(k)}
\]
its composition with $F^{(k)}$. Thus $\gamma$ associates to
$(t,y)\in S\times\PP^1$ the effective zero-cycle $ f_{t,*}\big(h_t^*(y)\big)  $ of degree $k$
on $X$.

Let
\[
Z\ =\ \overline{\gamma(S\times\PP^1)}\ \subseteq\ X^{(k)}
\]
be the closure of the image of $\gamma$. For general $t\in S$, the curve
$\gamma_t:\PP^1\lra Z$ is non-constant, since $f_t$ is generically
one-to-one. The  $\gamma_t(\PP^1)$ therefore form a covering family of
rational curves on $Z$.

Fix a smooth projective model $\widetilde Z\longrightarrow Z$ and,
after resolving the indeterminacies of its MRC fibration, write
\[
m:\widetilde Z\longrightarrow B
\]
for the resulting morphism. The strict transforms of the curves
$\gamma_t(\PP^1)$ form a covering family on $\widetilde Z$, and hence a very
general member of this family is contracted by $m$ (\cite[Theorem IV.5.2]{Kollar.Rational.Curves}). It follows that the
composite rational map
\[
S\times\PP^1\dashrightarrow\widetilde Z\longrightarrow B
\]
is constant on $\{t\}\times\PP^1$ for very general $t\in S$, and hence
factors through a dominant rational map
\[
\alpha:S\dashrightarrow B.
\]
After shrinking $S$ and $B$, we may suppose that $\alpha$ is a
morphism.  Moreover after shrinking $S$ further, we may  suppose in addition
that for every $t\in S$ the rational map
\[
\{t\}\times \PP^1 \dashrightarrow \widetilde Z
\]
is defined on a dense open subset of $\PP^1$, and that its image is
contained in the fibre $m^{-1}(\alpha(t))$.

Since the curves $\gamma_t(\PP^1)$ are non-constant and contracted by
$m$, one has
\begin{equation}
\label{Dim.B.Eqn}
\dim B\leq \dim Z-1\leq kn-1<p.
\end{equation}

\noi \textsc{Step 3}. The dimensional hypothesis enters through the
following observation.

\begin{quote}
\textsc{Claim.} For a general point $b\in B$, the curves $C_t$ with
\[
t\in S_b =_{\mathrm{def}} \alpha^{-1}(b)
\]
form a covering family on $X$.
\end{quote}
To prove the claim, let
\[
W\ \subseteq\ B\times X
\]
be the closure of the image of
\[
(\alpha\circ\pi,F):\CCC\longrightarrow B\times X.
\]
Thus $W$ is the closure of the locus of pairs $(b,x)$ such that
$x\in f_t(C_t)$ for some $t\in S_b$. If the claim were false, then the general fibre $W_b$ of
the projection
$
W\longrightarrow B
$
would be a proper subvariety of $X$. After shrinking $B$, we could
therefore suppose that
\[
\dim W_b\ \leq\ n-1
\]
for every $b\in B$.

Now consider the $p$-fold fibre product $W^p_B$ of $W$ over $B$.
The commutative square
\[
\xymatrix@C+10pt{
\CCC \ar[d]_\pi \ar[r]
    & W \ar[d] \\
S \ar[r]_\alpha
    & B
}
\]
induces a factorization
\[
F^p:\CCC^p_S\longrightarrow W^p_B\longrightarrow X^p.
\]
On the other hand, by \eqref{Dim.B.Eqn},
\[
\dim W^p_B
   \ \leq\
   \dim B+p(n-1)
   \ <\
   p+p(n-1)
   \ =\
   \dim X^p.
\]
It follows that $F^p$ cannot be dominant, contradicting the assumption
that the original family is $p$-connecting. This proves the claim.

\noi \textsc{Step 4}. We complete the proof by constructing the
varieties $R$ and $Y$. Fix a very general point $b\in B$ for which the
claim in Step 3 holds, and put
\[
R=m^{-1}(b),
\]
where as above $m : \widetilde{Z} \lra B$ represents the MRC fibration. 
We may suppose that $R$ is smooth and projective. The basic properties of the MRC fibration guarantee that  $R$ is rationally connected
(\cite[Theorems IV.3.10 and IV.5.2]{Kollar.Rational.Curves}).  
Since a general fibre of $h_t:C_t\lra\PP^1$ is reduced and $f_t$ is
generically one-to-one, the locus in $Z$ parametrizing reduced zero-cycles
is dense. Taking $b$ very general, we may therefore suppose that a general
point of $R$ parametrizes a reduced zero-cycle on $X$. Denote by
$R_{\rm red}\subseteq R$ the corresponding dense open subset.

By construction, $R$ comes with a morphism
$\rho:R\longrightarrow X^{(k)}.
$  Since the composite map
\[
S\times\PP^1\dashrightarrow\widetilde Z\longrightarrow B
\]
factors through $\alpha$, the zero-cycles arising from the pencils on
the curves $C_t$, for $t\in S_b$, are parametrized by $\rho(R)$. It
therefore follows from Step 3 that the points occurring in the
zero-cycles parametrized by $\rho(R)$ cover $X$.

Consider now the natural  degree $k$ finite morphism
\[
 X\times X^{(k-1)}\longrightarrow X^{(k)},
\qquad
(x,\xi)\longmapsto x+\xi.
\]
Taking the fibre product with $\rho : R \lra X^{(k)},$ we construct a variety $Y_1$ sitting in a diagram
\[
\xymatrix{
 Y_1 \ar[r]^-{\sigma'} \ar[d]_{\tau'}
    & X\\
R.
}
\]
To complete the proof, it remains only to show  that there is an irreducible component $Y$ of $Y_1$ that
dominates both $R$ and $X$. 

To this end, observe that  over the open subset
$R_{\rm red}\subseteq R$ the morphism $\tau'$ is finite etale of degree
$k$. Let 
\[ E_1 \, , \,  \ldots \, , \, E_s\ \subseteq\  Y_1\] be the closures of the irreducible
components of $(\tau')^{-1}(R_{\rm red})$. Then each $E_i$ dominates $R$,
and their degrees over $R$ sum to $k$.
We assert that at least one of the $E_i$ dominates $X$. If not, there is a proper closed subset $H \subsetneq X$ such that  for every $u\in R_{\rm red} $ the cycle $\rho(u)$ is supported on $H$.
Denoting by \[X^{(k)}_H \ \subseteq \ X^{(k)}\] the proper closed subset of cycles supported on $H$,  this means that  $\rho( R_{\rm red}) \subseteq X^{(k)}_H $. Since $R_{\rm red} \subseteq R$ is dense,   this would imply that
  \[ \rho(R) \ \subseteq \ X^{(k)}_H . \] But Step 3
shows that the points occurring in these cycles are dense in $X$, a
contradiction. Thus some component $Y=E_i$ dominates $X$ as well as
$R$, and $\deg(Y \lra R) \le k$, so we are done.
\end{proof}

Finally, we turn to the proof of Corollary
\ref{Rat.Conn.Covering.Ex}. Given the output of Theorem \ref{Intro.Eventual.Gonality.Thm}, the plan is to start with a rational curve connecting many points of $R$ and then consider first its preimage in $Y$, and subsequently its projection to $X$. The one point requiring   care is to ensure that the curves so constructed remain irreducible. To this end we assemble some  results
of Koll\'ar from \cite{KFG}.

\begin{lemma} 
\label{RC.Cover.Pullback.Lemma}
Let $R$ be a smooth projective rationally connected variety, and let
\[
\tau:Y\longrightarrow R
\]
be a dominant generically finite morphism of degree $e$, with $Y$
irreducible. Given $p$ general points
$
r_1,\ldots,r_p\in R,
$
there is a morphism
\[
f:\PP^1\longrightarrow R
\]
whose image contains the $r_i$ and such that
\[
Y\times_R\PP^1
\]
is irreducible.
\end{lemma}

\begin{proof}[Sketch of Proof]
Consider the family
\[
F:W\times\PP^1\longrightarrow R
\]
produced by Koll\'ar in \cite[proof of Theorem 6, paragraph (28), p. 87]{KFG}. By the freedom in its construction noted in
\cite[Remark~4(6)]{KFG}, we may suppose that, for every $w\in W$,
\[
f_w^*T_R\otimes\OO_{\PP^1}(-p)
\quad\text{is ample},
\]
where
$
f_w=F|_{\{w\}\times\PP^1}.
$
Corollary~7 of \cite{KFG}, applied to $Y\lra R$, gives a dense open
subset
$
W^0\subseteq W
$
such that
\[
Y\times_R\PP^1_w
\]
is irreducible for every $w\in W^0$.

  Consider the $p$-fold evaluation morphism
\[
\ev_p:W\times(\PP^1)^p \longrightarrow R^p.
\]
By the usual deformation-theoretic calculations, its differential in the map directions contains the natural evaluation
homomorphism
\[
H^0(\PP^1,f_w^*T_R)
\longrightarrow
\bigoplus_{i=1}^p T_{R,f_w(z_i)}.
\]
This homomorphism is surjective thanks to the vanishing
\[
H^1\!\left(
\PP^1,f_w^*T_R(-z_1-\cdots-z_p)
\right)=0,
\]
which follows from the positivity imposed above. Therefore $\ev_p$ is
dominant.

The restriction of $\ev_p$ to the dense open subset
\[
W^0\times(\PP^1)^p
\]
remains dominant. Hence, for a general $p$-tuple
$(r_1,\ldots,r_p)\in R^p$, there are $w\in W^0$ and pairwise distinct
$z_1,\ldots,z_p\in\PP^1$ such that
\[
f_w(z_i)=r_i
\qquad (1\leq i\leq p).
\]
The corresponding fibre product
\[
Y\times_R\PP^1_w
\]
is irreducible by the choice of $W^0$, proving the lemma.
\end{proof}

\begin{proof}[Proof of Corollary \ref{Rat.Conn.Covering.Ex}]
Let $c_0$ denote the minimum appearing on the right-hand side of the
statement. Thus there is a diagram
\[
\xymatrix{
Y \ar[r]^\sigma \ar[d]_\tau & X\\
R
}
\]
in which $\sigma$ is dominant, $R$ is rationally connected, and $\tau$
is generically finite of degree $c_0$. We will construct for every $p$ a $p$-connecting family of curves having gonality $\le c_0$.

Fix $p$ general points
\[
q_1,\ldots,q_p\in X.
\]
Choose general lifts $y_i\in Y$ with $\sigma(y_i)=q_i$, and put
\[
r_i=\tau(y_i)\in R.
\]
Since the points $q_i$ are general and both maps in the diagram are
dominant, we may suppose that the $y_i$, and
hence the $r_i$, are general.

Applying Lemma \ref{RC.Cover.Pullback.Lemma}, we find a morphism
$
f:\PP^1\longrightarrow R
$
passing through the $r_i$ such that
\[
\Gamma \ = \ Y\times_R\PP^1
\]
is irreducible. If $f(z_i)=r_i$, then
$
(y_i,z_i)\in\Gamma,
$
so the image of $\Gamma$ in $X$ contains all the points $q_i$.
The normalization of a projective completion of $\Gamma$ admits a
degree $c_0$ map to $\PP^1$. Since gonality doesn't increase under coverings of curves,\footnote{If $\Gamma \lra \Delta$ is a branched covering of curves, then a $g^1_k$ on $\Gamma$ pushes down to a (possibly non-linear) one-dimensional family of linearly equivalent effective divisors of degree $k$ on $\Delta$, which must therefore have gonality $\le k$.} its image in $X$ also has gonality at
most $c_0$. Thus
$\pconngon(X)\leq c_0$
for every $p\geq2$.

Conversely, let $k$ be the eventual value of $\pconngon(X)$, and choose
$p$ in the stable range sufficiently large that
$
p\geq kn.
$
Theorem \ref{Intro.Eventual.Gonality.Thm} then produces a diagram in
which the covering of the rationally connected variety has degree at
most $k$. This gives the reverse inequality
$
c_0\leq k,
$
and we are done.
\end{proof}

\begin{problem} It would be interesting to estimate or compute the final value of the $p$-connecting gonality for natural choices of $X$. When $X = \mathcal{M}_g$ is the moduli space of curves of genus $g$, suitable Hurwitz spaces fit naturally into the setting of Corollary \ref{Rat.Conn.Covering.Ex}, but presumably the resulting covering degree is far from optimal. 
\end{problem}

\begin{problem}
One might wonder whether there an analogue of the theorem of Graber--Harris--Starr \cite{GHS} in the setting of Corollary \ref{Rat.Conn.Covering.Ex}.
\end{problem}

\section{Hypersurfaces} \label{Hypersurfaces}

As an illustration of the general theory, we compute in this section the asymptotics of these invariants for  hypersurfaces.

Consider then a smooth hypersurface 
\[
X \ = \ X_d \ \subseteq \ \PP^{n+1}
\] 
of degree $d$ and dimension $n$. To begin with, we observe the elementary

\begin{proposition}
For any $p\ge 2$, the $p$-connecting gonality of $X$ satisfies
\[
d \,-\, n \ \le \ \pconngon(X) \ \le \ d-1.
\]
In particular, 
\[
\pconngon(X)\,\sim\,d \ \ \text{as } d\to\infty.
\]
\end{proposition}

\begin{proof}
The lower bound holds for any covering family: see for instance \cite{BDELU}. On the other hand, projection from a point of $X$ gives a rational covering $X \dra \PP^n$ of degree $d-1$, and so $\pconngon(X)\le d-1$ thanks to the results of the previous section.
\end{proof}

\begin{remark}
Assume that $X$ is very general of degree $d\ge 2n+2$. In view of \cite{BDELU} and \cite{Yang}, it seems likely that $\pconngon(X)=d-1$ for $p\gg 0$, but we have not tried to verify this. \qed
\end{remark}

The $p$-connecting genus is more interesting because one would like to understand its joint asymptotic behavior in $p$ and $d$. To this end, fix $n\geq 2$ and define
\[
r(p)=r_n(p)
=_{\mathrm{def}}
\min\left\{
r\geq 1
\ \middle|\
\binom{r+n+1}{n+1}\geq p+n-1
\right\}.
\]
Thus
\[
r(p)\sim \bigl((n+1)!\,p\bigr)^{\frac1{n+1}}
\qquad\text{as }p\to\infty;
\]
in particular,
\[
r(p)\asymp_n p^{\frac1{n+1}}.
\]

The result is then the following.

\begin{theorem}\label{Precise.Hypersurface.Thm}
Fix an integer $n\geq 2$. There exist positive constants $b_n$ and $B_n$, depending only on $n$, such that for every smooth hypersurface
\[
X_d\ \subseteq\ \PP^{n+1}
\]
and all integers $p\geq 3$ and $d\geq n+3$ satisfying $r(p)<d$, one has
\begin{equation}\label{Joint.Asympt.Eqn}
b_n\cdot p^{\frac{n-1}{n+1}}d^2
\ \leq\
\pconngenus(X_d)
\ \leq\
B_n\cdot p^{\frac{n-1}{n+1}}d^2.
\end{equation}
In other words,
\[
\pconngenus(X_d)
\ \in\
\Theta_n\left(
p^{\frac{n-1}{n+1}}\cdot d^2
\right)
\]
when $p$ and $d$ are in the indicated range.
\end{theorem}

\begin{remark}
Concerning the regime in which these asymptotics hold, observe that
\[
\frac{p^{\frac{n-1}{n+1}}d^2}{p}
\ =\
\left(\frac{d^{n+1}}{p}\right)^{\frac{2}{n+1}}.
\]
Hence if $p/d^{n+1}\to\infty$, then
\[
p^{\frac{n-1}{n+1}}d^2=o(p).
\]
On the other hand, Theorem~\ref{Intro.Thm.B} gives a linear lower bound in $p$, and so the upper bound in \eqref{Joint.Asympt.Eqn} cannot hold when $p\gg d^{n+1}$. \qed
\end{remark}

The plan is to construct the required curves as suitable complete intersections in $X$. We start with an elementary remark.

\begin{lemma}\label{Bertini.Lemma}
Let $X\subseteq\PP^N$ be a smooth non-degenerate subvariety of dimension $n$, and fix general points $q_1,\ldots,q_p\in X$. If
\[
p\leq N-n+2,
\]
then the $q_i$ lie on a smooth codimension $n-1$ linear section of $X$.
\end{lemma}

\begin{proof}
By enlarging the set of points, it is enough to treat the case $p=m=_{\mathrm{def}}N-n+2$. Let $\GGGG$ be the Grassmannian of codimension $n-1$ linear spaces in $\PP^N$, so that $\dim\GGGG=(N-n+2)(n-1)=m(n-1)$. For general $\Lambda\in\GGGG$, Bertini's theorem asserts that
\[
C_\Lambda=X\cap\Lambda
\]
is a smooth curve. Consider the incidence variety parametrizing a linear space $\Lambda\in\GGGG$ together with $m$ ordered points on $C_\Lambda$. Its dimension is
\[
\dim\GGGG+m=m(n-1)+m=mn=\dim X^m.
\]
Moreover, $m$ general points on the non-degenerate curve $C_\Lambda\subseteq\Lambda\cong\PP^{m-1}$ span $\Lambda$. Hence the natural map from this incidence variety to $X^m$ is generically finite and dominant. Therefore the set of $m$-tuples for which the assertion fails is contained in a proper subset of $X^m$.
\end{proof}

\begin{proof}[Proof of Theorem \ref{Precise.Hypersurface.Thm}]
We start by constructing $p$-connecting curves whose genus approaches the upper bound in the statement.
Set
$N_m \ = \ \hh{0}{\PP^{n+1}}{\OO_{\PP^{n+1}}(m)} - 1$. Then $r = r(p)$ is the least integer having the property that $p \le N_r - n + 2$. Now choose general points 
\[q_1, \ldots, q_p \in X,
\]
and consider the $r$-fold Veronese embedding 
\[ \nu_r \ : \ X \ \hookrightarrow \ \PP^{N_r}.
\]
By the Lemma, a general linear space of codimension $n-1$ in $\PP^{N_r} $ through the $q_i$ will cut out a smooth curve  on the Veronese image of $X$ passing through the $q_i$. Going back to the original embedding $X \subseteq \PP^{n+1}$ the curve $C$ just constructed is the complete intersection in $\PP^{n+1}$ of $n-1$ hypersurfaces of degree $r$ and the degree $d$ hypersurface $X$. Thus $C$ has degree $r^{n-1}d$, and by adjunction its genus $g$ satisfies
\begin{align*}
2g - 2 \ &= \ \big( (n-1)r + d - (n+2)\big) \cdot r^{n-1}d \\
&\le \  n\cdot r^{n-1}d^2,
\end{align*}
where we have used the assumption $r<d$. Via the approximation $r(p)\asymp_n p^{\frac1{n+1}}$, we see that this  bound has the asserted shape $p^{\frac{n-1}{n+1}}d^2$.

We can give a more precise bookkeeping of the dependencies among the quantities in play.  By the minimality of $r$,
$
\binom{r+n}{n+1}<p+n-1.
$
On the other hand,
$
\binom{r+n}{n+1}\geq r^{n+1}/(n+1)!,
$
while $p+n-1\leq np$ since $n\geq2$ and $p\geq3$. Therefore
\[
r\ \leq\ A_n\cdot p^{\frac1{n+1}},
\quad \text{ where } \quad
A_n=_{\mathrm{def}}\bigl(n(n+1)!\bigr)^{\frac1{n+1}}.
\]
It follows that
\[
g
\ \leq\
1+\frac{n}{2}A_n^{n-1}
p^{\frac{n-1}{n+1}}d^2.
\]
Since $p\geq3$ and $d\geq n+3$, the additive constant can be absorbed into the main term. This proves the upper bound in \eqref{Joint.Asympt.Eqn}, with a constant $B_n$ depending only on $n$.

We now turn to the lower bound. Let $C$ be a general member of any $p$-connecting family on $X$, and write
$
g=g(C)
$
and
$
e=\deg_{\PP^{n+1}}C. 
$
The elementary inequality for covering families established in \cite{BDELU} gives
\[
2g-2\ \geq\ K_X\cdot C
       \ =\ (d-n-2)e. \tag{*}
\]
We will combine this with an upper bound on $g$ in terms of $e$.

Consider the image of $C$ under the $r$-fold Veronese embedding: call it
\[ \Gamma \ =_\text{def} \ \nu_r(C) \ \subseteq \ \PP^{N_r}.
\]
Since $r<d$, the degree $r$ hypersurfaces on $\PP^{n+1}$ restrict to the complete linear series $\vert\OO_X(r)\vert$. Thanks to our choice of $r$, the images under $\nu_r$ of $p$ general points of $X$ are linearly independent. Since $C$ passes through these points, it follows that $\Gamma$ spans a linear space of dimension $s \ge p-1$.
Moreover, $\deg\Gamma=re$. After a general projection from its linear span to $\PP^{p-1}$, the curve $\Gamma$ maps birationally onto a non-degenerate curve of degree $re$. Castelnuovo's bound on the genus of a non-degenerate curve in $\PP^{p-1}$ (\cite[p.~527]{GH}) therefore gives
\[
g
\ \leq\
\frac{(re-1)^2}{2(p-2)}
\ \leq\
\frac{r^2e^2}{2(p-2)}. \tag{**}
\]
Combine (*) and (**) to get
\[
(d-n-2)e
\ \leq\
2g-2
\ <\
2g
\ \leq\
\frac{r^2e^2}{p-2}.
\]
Consequently
\[
e
\ \geq\
\frac{(d-n-2)(p-2)}{r^2}.
\]
Substituting back into (*) yields
\[
g
\ \geq\
1+\frac{(d-n-2)^2(p-2)}{2r^2}.
\]
Using again that $r(p)\asymp_n p^{\frac1{n+1}}$, we see that this is bounded below by a constant multiple of $\frac{d^2 p}{r^2} \asymp  p^{\frac{n-1}{n+1}}d^2 $.

It remains only to carry out the precise accounting underlying this lower bound. The assumptions $d\geq n+3$ and $p\geq3$ imply
$
d-n-2\geq d/(n+3)
$
and
$
p-2\geq p/3.
$
Using also $r\leq A_np^{1/(n+1)}$, we obtain
\[
g
\ \geq\
\frac{1}{6(n+3)^2A_n^2}\,
p^{\frac{n-1}{n+1}}d^2.
\]
Thus the lower bound in \eqref{Joint.Asympt.Eqn} holds, for example with
$
b_n=1/\bigl(6(n+3)^2A_n^2\bigr).
$
This completes the proof.
\end{proof}

\section{Coda: Higher-dimensional connecting families}
\label{Higher.Dim.Coverings}

This section is devoted to bounds on the volume of a family of
$k$-dimensional subvarieties of $X$ passing through $p$ general points.
As above, $X$ is a smooth projective variety of dimension $n$, and we
assume that $1 \leq k \leq n-1$.

Consider a family
\begin{equation} \label{Fam.of.Subvars}
\vcenter{\hbox{$
\xymatrix{
\VVV \ar[r]^{F} \ar[d]_\pi & X \\
S
}
$}}
\end{equation}
of smooth $k$-dimensional projective varieties $\{V_t\}_{t\in S}$ mapping to $X$. We
assume that each $V_t$ maps birationally onto its image, and that the
induced map
\begin{equation}
F^p:\VVV^p_{/S}\lra X^p
\end{equation}
is dominant. Recall that the \textit{canonical volume} of $V=V_t$ is
defined to be
\[
\vol(V)
   =\lim_{m\to\infty}
      \frac{\hh{0}{V}{mK_V}}{m^k/k!}.
\]
It follows from the deformation invariance of plurigenera
\cite{S} that $\vol(V_t)$ is constant in smooth projective families.

The main point is that these volumes satisfy a linear lower bound in $p$.

\begin{theorem} \label{k-dim.linear.lower.bound}
Assume that $X$ is of general type. Then, in the situation above, there
exists a positive real number $a_{k,X}>0$ such that
\[
\vol(V)\geq a_{k,X}\cdot(p-k).
\]
\end{theorem}

\noi This implies the first statement in Theorem
\ref{k-dim.p-conn.Thm} from the Introduction. To produce $p$-connecting
families whose volume is bounded above linearly in $p$, one can use
suitable complete intersections as in the proof of Theorem
\ref{Precise.Hypersurface.Thm}; see Remark
\ref{k-Dim.Lin.Upper.Bound}.

\begin{proof}[Proof of Theorem \ref{k-dim.linear.lower.bound}]
Fix a general point
\[
(t;q_1,\ldots,q_p)\in \VVV^p_{/S}
\]
at which $F^p$ is smooth, and put $V=V_t$ and $f=F|_V$. As in
\S\ref{Section.Connecting.Genus}, the differential of $F^p$, modulo the
tangent spaces to the fibres, gives a natural map
\begin{equation} \label{Map.to.Nf}
T_tS\lra\bigoplus_{i=1}^p N_f|_{q_i},
\end{equation}
where
$
N_f=\coker\big(T_V\lra f^*T_X\big)
$
is the normal sheaf to $f$. We may suppose that $f$ has rank $k$ at each
$q_i$. Thus $N_f$ is locally free of rank
\[
e=_{\mathrm{def}}n-k
\]
near the $q_i$, and the smoothness of $F^p$ implies that
\eqref{Map.to.Nf} is surjective. We can therefore choose, for every
$1\leq i\leq p$, vectors
\[
v_{i,1},\ldots,v_{i,e}\in T_tS
\]
whose images form a basis of $N_f|_{q_i}$ and vanish in $N_f|_{q_j}$ for
$j\neq i$.

Put
\[
L_f=K_V\otimes f^*K_X^{-1}.
\]
There is a canonical homomorphism
\[
\bigwedge^e N_f
\ \lra\
\det(f^*T_X)\otimes\det(T_V)^{-1}
\ =\
L_f,
\]
so we arrive at a homomorphism
\[
\bigwedge^e T_tS\lra H^0(V,L_f).
\]
(See \cite[\S 2]{LMP} for the details of a closely related
construction.) The element
$
\xi_i=v_{i,1}\wedge\cdots\wedge v_{i,e}
$
maps to a section $\delta_i\in H^0(V,L_f)$ satisfying
\[
\delta_i(q_i)\neq 0,
\qquad
\delta_i(q_j)=0\quad\text{for }j\neq i.
\]
Consequently, the sections $\delta_1,\ldots,\delta_p$ are linearly
independent, and hence
\begin{equation} \label{Crucial.Lower.Bound.L_f}
\hh{0}{V}{L_f}\geq p.
\end{equation}

We now argue that the theorem follows formally from
\eqref{Crucial.Lower.Bound.L_f}. Since $K_X$ is big, we may fix, once
and for all, a positive integer $b=b_X$ and a decomposition
\[
bK_X \lin H+E,
\]
where $H$ is an effective ample divisor and $E$ is effective, and these
choices are independent of $p$. Since the family $\VVV$ covers $X$, we
may suppose that the image of $f$ is not contained in
$\Supp(H+E)$. Then $f^*(bK_X)$ is represented by an effective and big
divisor on $V$. Now observe that
\[
bK_V
\ \lin\
bL_f+f^*(bK_X)
\ \lin\
L_f+\big((b-1)L_f+f^*(bK_X)\big).
\tag{*}
\]
Since $L_f$ is effective, it follows that
$
bK_V\lin L_f+G,
$
where $G$ is effective and big. Hence we see in the first place that
$V$ is of general type. Moreover, \eqref{Crucial.Lower.Bound.L_f} and
(*) imply that
$
h^0(bK_V)\geq p,
$
and hence also
\begin{equation} \label{Lower.Bound.bKV}
\hh{0}{V}{mbK_V}
\ \geq\
p
\end{equation}
for every integer $m>0$.

We next invoke the uniform pluricanonical birationality theorem of
Hacon--McKernan, Takayama, and Tsuji \cite{HM,Takayama}. This result asserts
that there exists an integer $m_k$, depending only on $k$, such that
$\linser{\ell K_W}$ defines a birational map for every smooth projective
$k$-fold $W$ of general type and every $\ell\geq m_k$. Put
$
M \ = \ bm_k.
$
Then $\linser{MK_V}$ defines a birational map. Moreover, taking
$m=m_k$ in \eqref{Lower.Bound.bKV}, we have
\[
\hh{0}{V}{MK_V}
\ \geq\
p.
\]

Resolving the base locus of $\linser{MK_V}$ yields a birational morphism
$\mu:V^\pr\lra V$ and a decomposition
\[
\mu^*\linser{MK_V}
\ =\
\linser{B}+D,
\]
where $D$ is effective and $\linser{B}$ is base-point free. Thus
$\linser{B}$ defines a morphism
\[
\phi
\ =\
\phi_{\linser{B}}
:
V^\pr\lra\PP^{r(B)}
\]
that is birational onto its image. We have
\[
h^0(V^\pr,B)
\ =\
h^0(V,MK_V)
\ \geq\
p,
\]
so $r(B)\geq p-1$. Therefore
\[
Z
\ =_{\mathrm{def}}\
\Image(\phi)
\ \subseteq\
\PP^{r(B)}
\]
is a non-degenerate $k$-dimensional subvariety whose degree is
$(B^k)_{V^\pr}$. The elementary degree--codimension inequality gives
\[
(B^k)_{V^\pr}
\ =\
\deg Z
\ \geq\
\codim_{\PP^{r(B)}}(Z)+1
\ =\
r(B)-k+1
\ \geq\
p-k.
\]
On the other hand,
\[
(B^k)_{V^\pr}
\ =\
\vol(B)
\ \leq\
\vol\big(\mu^*(MK_V)\big)
\ =\
M^k\vol(K_V).
\]
Combining these inequalities, and recalling that $M=bm_k$, we find
\[
\vol(V)
\ \geq\
\frac{p-k}{(bm_k)^k}.
\]
Thus we can take $a_{k,X}=1/(bm_k)^k$, completing the proof.
\end{proof}

\begin{remark} [\textbf{Bounds on geometric genus}]
We refer to \cite[Theorem B]{LMP} for lower bounds on the geometric genus  $h^0(K_V)$.
\end{remark}

\begin{remark}[\textbf{Construction of families with linear volume growth}]
\label{k-Dim.Lin.Upper.Bound}
For completeness, we indicate briefly the construction of connecting
subvarieties whose volumes are bounded above linearly in $p$. Fix a
very ample divisor $A$ on $X$, sufficiently positive that $K_X+A$ is
very ample, and put
\[
N_m
\ =\
\hh{0}{X}{mA}-1.
\]
Choose $m$ minimally so that
\[
p
\ \leq\
N_m-n+k+1.
\]
The same incidence argument used in Lemma~\ref{Bertini.Lemma} shows that
$p$ general points of $X$ lie on a smooth codimension $n-k$ linear
section of the embedding defined by $\linser{mA}$. Its inverse image
$V_m\subseteq X$ is a smooth complete intersection of $n-k$ divisors in
$\linser{mA}$. Since
\[
N_m
\ =\
\frac{(A^n)}{n!}m^n+O(m^{n-1}),
\]
the minimality of $m$ gives $m^n=O(p)$. On the other hand, adjunction
yields
\[
\vol(V_m)
\ =\
\bigl(K_X+(n-k)mA\bigr)^k\cdot(mA)^{n-k}
\ =\
O_{k,X}(m^n)
\ =\
O_{k,X}(p).
\qquad\qed
\]
\end{remark}

\begin{remark}[\textbf{Analogue of connecting gonality}]
\label{Higher.Dim.Gonality}
The $p$-connecting canonical volume is the natural generalization to
higher dimensions of the $p$-connecting genus for families of curves.
However, it is not clear to us whether there is a good analogue of
connecting gonality. One might of course consider the degree of
irrationality $\irrdeg(V_t)$ of a connecting family of subvarieties, or
some variant thereof, but it does not seem immediately obvious that the
picture from the curve case extends to subvarieties of dimension
$k>1$. \qed
\end{remark}

Finally, we indicate the extension to the present setting of the
asymptotic computations for hypersurfaces in
\S\ref{Hypersurfaces}. Consider then a smooth hypersurface
\[
X
\ =\
X_d
\ \subseteq\
\PP^{n+1}
\]
of degree $d$. We take
$
r(p)\asymp_n p^{\frac1{n+1}}
$
as in \S\ref{Hypersurfaces}.

\begin{theorem}\label{Higher.Dim.Hypsf.Asymptotics}
Fix $n$ and $k$ with $1\leq k<n$. There are constants
$c_{n,k},C_{n,k}>0$ such that if $p$ is sufficiently large in terms of
$k$, $d\geq n+3$, and $r(p)<d$, then the least canonical volume of a
$p$-connecting family of $k$-dimensional subvarieties of $X_d$ satisfies
\[
c_{n,k}\,p^{\frac{n-k}{n+1}}d^{k+1}
\ \leq\
\vol(V)
\ \leq\
C_{n,k}\,p^{\frac{n-k}{n+1}}d^{k+1}.
\]
In other words, in this range the minimum lies in
\[
\Theta_{n,k}\left(
p^{\frac{n-k}{n+1}}d^{k+1}
\right).
\]
\end{theorem}

The proof follows closely the arguments from
\S\ref{Hypersurfaces}. The most interesting new ingredient is an
analogue of Castelnuovo's bound for the volume of an arbitrary
non-degenerate subvariety of projective space.

\begin{proposition}[\textbf{Higher-dimensional Castelnuovo bound}]
\label{cast}
Let
\[
W
\ \subseteq\
\PP^N
\]
be a non-degenerate subvariety of degree $\delta$ and dimension $k<N$. Then
\[
\vol(W)
\ \leq\
\frac{\delta^{k+1}}{(N-k)^k}.
\]
\end{proposition}
 
\begin{proof}
When $W$ is smooth and $K_W$ is nef, related and more precise bounds on
the self-intersection number $(K_W^k)_W$ were established by Di Gennaro
\cite{DiGennaro}. The present argument mirrors his, but avoids these
extra hypotheses by resolving $W$ and passing to a minimal model. The
idea is to reduce to the classical bound for curves by taking general
hyperplane sections and then applying the inequalities of
Khovanskii--Teissier.

Turning to details, we may assume that $W$ is of general type, since
otherwise $\vol(W)=0$ and the assertion is trivially true. Let
\[
\mu
\, : \,
W^\pr
\lra
W
\]
be a resolution of singularities, and let $W_{\min}$ be a minimal model
of $W^\pr$, which exists thanks to \cite{BCHM}. Choose a smooth variety $Z$ resolving the resulting
birational map, so that we have a diagram
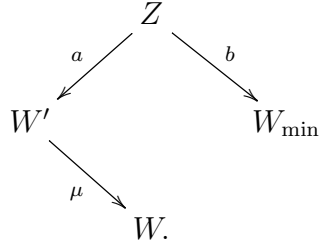
\[
\xymatrix{
& Z \ar[dl]_a \ar[dr]^b & \\
W^\pr \ar[dr]_\mu && W_{\min} \\
& W. &
}
\]
Put
\[
H
\ =_{\mathrm{def}}\
(\mu\circ a)^*\OO_W(1)
\qquad\text{and}\qquad
P
\ =_{\mathrm{def}}\
b^*K_{W_{\min}}.
\]
Thus $P$ is a nef and big $\QQ$-divisor on $Z$, while $H$ is globally
generated, and hence nef, as well as big. We have
\[
(H^k)_Z
\ =\
\delta,
\qquad
(P^k)_Z
\ =\
\vol(W).
\]
Moreover, since $W_{\min}$ has terminal singularities, one has
\[
K_Z
\ =\
P+E
\]
for an effective $b$-exceptional $\QQ$-divisor $E$.

Let
\[
C^\pr
\ =\
H_1\cap\cdots\cap H_{k-1}
\]
be the intersection of $k-1$ general members of $\linser{H}$. Then
$C^\pr$ is a smooth curve mapping birationally onto a general linear
curve section
\[
C
\ \subseteq\
W
\ \subseteq\
\PP^N.
\]
Thus $C$ is non-degenerate of degree $\delta$ in
$
\PP^{N-k+1}
$
and has geometric genus $g=g(C^\pr)$. Adjunction gives
\[
(K_Z\cdot H^{k-1})_Z
\ =\
2g-2-(k-1)\delta.
\]
Since $K_Z=P+E$, with $E$  effective, and $H$ is nef, it follows that
\begin{equation}\label{P.H.Bound}
(P\cdot H^{k-1})_Z
\ \leq\
2g-2-(k-1)\delta
\ \leq\
2g-2.
\end{equation}

Put $c=N-k$. Castelnuovo's bound, applied to the non-degenerate curve
$C\subseteq\PP^{c+1}$, gives
\[
2g-2
\ \leq\
\frac{\delta^2}{c}
\ =\
\frac{\delta^2}{N-k}.
\]
Thanks to \eqref{P.H.Bound}, this implies
\[
(P\cdot H^{k-1})_Z
\ \leq\
\frac{\delta^2}{N-k}.
\]
Finally, the Khovanskii--Teissier inequality for the nef divisors $P$
and $H$ on $Z$ (cf.~\cite[Corollary 1.6.3]{PAG}) yields
\[
\big((P\cdot H^{k-1})_Z\big)^k
\ \geq\
(P^k)_Z\cdot\big((H^k)_Z\big)^{k-1}
\ =\
\vol(W)\cdot \delta^{k-1}.
\]
Therefore
\[
\vol(W)
\ \leq\
\frac{\big((P\cdot H^{k-1})_Z\big)^k}{\delta^{k-1}}
\ \leq\
\frac{1}{\delta^{k-1}}
\left(\frac{\delta^2}{N-k}\right)^k
\ =\
\frac{\delta^{k+1}}{(N-k)^k},
\]
as required.
\end{proof}

\begin{proof}[Sketch of Proof of Theorem
\ref{Higher.Dim.Hypsf.Asymptotics}]
The argument is essentially the same as in the proof of Theorem
\ref{Precise.Hypersurface.Thm}. For the upper bound, take $V$ to be a
smooth complete intersection of $n-k$ divisors in
$\linser{\OO_X(r)}$ passing through the prescribed points, where
$r=r(p)$. Adjunction gives
\[
\vol(V)
\ =\
\big(d+(n-k)r-(n+2)\big)^k r^{n-k}d
\ \ll_{n,k}\
r^{n-k}d^{k+1}.
\]
Since $r\asymp_n p^{1/(n+1)}$, this proves the upper bound.

For the lower bound, let $V$ be a general member of any $p$-connecting
family, and put
$
e=\deg(V).
$
The effectivity of $K_V-f^*K_{X_d}$ established in the proof of
Theorem \ref{k-dim.linear.lower.bound}, together with the monotonicity
of volume, gives
\[
\vol(V)
\ \geq\
(d-n-2)^k e.
\]
On the other hand, the $r$-fold Veronese image of $V$ spans a projective
space of dimension at least $p-1$. After a suitable birational
projection to $\PP^{p-1}$, Proposition \ref{cast} yields
\[
\vol(V)
\ \leq\
\frac{(r^ke)^{k+1}}{(p-k-1)^k}.
\]
Combining these inequalities gives
\[
\vol(V)
\ \geq\
\frac{(d-n-2)^{k+1}(p-k-1)}{r^{k+1}}.
\]
Since $d-n-2\geq d/(n+3)$, while $p-k-1$ is bounded below by a
positive multiple of $p$ when $p$ is sufficiently large in terms of
$k$, the relation $r\asymp_n p^{1/(n+1)}$ yields
\[
\vol(V)
\ \gg_{n,k}\
p^{\frac{n-k}{n+1}}d^{k+1},
\]
as required.
\end{proof}

 %
 %
 %

 \end{document}